\documentclass[a4paper,12pt]{article}
\usepackage{a4wide}
\usepackage{amsmath}
\usepackage{amssymb}
\usepackage{amsthm}
\usepackage{latexsym}
\usepackage{graphicx}
\usepackage[english]{babel}
\usepackage{makeidx}

\newtheorem{obs} [subsection]{Remark}
\newtheorem{exm} [subsection]{Example}

\newtheorem{conj}[subsection]{Conjecture}
\newtheorem{teor}[subsection]{Theorem}
\newtheorem{lema}[subsection]{Lemma}
\newtheorem{cor} [subsection]{Corollary}

\def\sdepth{\operatorname{sdepth}}
\def\depth{\operatorname{depth}}

\begin{document}
\selectlanguage{english}
\frenchspacing

\large
\begin{center}
\textbf{Stanley depth of complete intersection monomial ideals}

Mircea Cimpoea\c s
\end{center}
\normalsize

\begin{abstract}
We compute the Stanley depth of irreducible monomial ideals and we show that the Stanley depth of a monomial complete intersection ideal is the same as the Stanley depth of it's radical. Also, we give some bounds for the Stanley depth of a monomial complete intersection ideal.

\vspace{5 pt} \noindent \textbf{Keywords:} Stanley depth, monomial ideal, complete intersection.

\vspace{5 pt} \noindent \textbf{2000 Mathematics Subject
Classification:}Primary: 13H10, Secondary: 13P10.
\end{abstract}

\section*{Introduction}

Let $K$ be a field and $S=K[x_1,\ldots,x_n]$ the polynomial ring over $K$.
Let $I\subset S$ be a monomial ideal. A Stanley decomposition of $I$, is a decomposition $\mathcal D: I=\bigoplus_{i=1}^ru_i K[Z_i]$, where $u_i\in S$ are monomials and $Z_i\subset\{x_1,\ldots,x_n\}$. We denote $\sdepth(\mathcal D)=\min_{i=1}^r |Z_i|$ and
$\sdepth(I)=\max\{\sdepth(\mathcal D)|\;\mathcal D$ is a Stanley decomposition of $I \}$. The number $\sdepth(I)$ is called the \emph{Stanley depth} of $I$. Stanley's conjecture says that $\sdepth(I)\geq \depth(I)$. The Stanley conjecture was proved 
for $n\leq 5$ and in other special cases, but it remains open in the general case. See for instance, \cite{apel}, \cite{hsy}, \cite{jah}, \cite{pops} and \cite{popi}.

Let $I\subset S$ be a monomial ideal. We assume that $G(I)=(v_1,\ldots,v_m)$, where $G(I)$ is the set of minimal monomial generators of $I$. We recall some notations from \cite{hvz}. Let $g=(g(1),\ldots,g(n))\in \mathbb N^n$ with $v_i|x^g$
for all $i=1,\ldots,m$, where $x^g:=x_1^{g(1)}\cdots x_n^{g(n)}$. We denote $P_I^{g}=\{\sigma\in\mathbb N^n|\; \sigma\leq g$ and $v_i|x^{\sigma}$ for some $i\}$. If $P_I^g = \bigcup_{i=1}^r [c_i,d_i]$ is a partition of $P_I^g$, we denote 
$\rho(d_i)=\#\{j|\; d_i(j) = g(j)\}$. 

Herzog, Vladoiu, Zheng proved in \cite[Theorem 2.4]{hvz} that there exists a partition $\bigcup_{i=1}^r [c_i,d_i]$ of $P_I^{g}$ such that $\sdepth(I)=\min\{\rho(d_i):\;i\in [r]\}$. Another result from \cite{hvz} says that if $I\subset S$ is a monomial ideal and $IS[x_{n+1}]$ is the extension of $I$ in $S[x_{n+1}]$ then $\sdepth(IS[x_{n+1}])=\sdepth(I)+1$.
In \cite{hvz}, it was conjectured that $\sdepth((x_1,\ldots,x_n))=\left\lceil n/2 \right\rceil$. Herzog, Vladoiu and Zheng reduced this question to a purely combinatorial problem: to show that for any $n$, there exists a partition of $\mathcal P([n])\setminus \{\emptyset\} = \cup_{i=1}^r[X_i,Y_i]$ such that $|Y_i|\geq n/2$ for all $i$. Csaba Biro, David M.Howard, Mitchel T.Keller, William T.Trotter and Stephen J.Young give in \cite[Theorem 1.1]{par} a positive answer to this problem.

In the first section, we compute the Stanley depth of monomial irreducible ideals and show that this is the same as the Stanley depth of monomial prime ideals, see Theorem $1.3$. In the second section we show that the Stanley depth of a monomial complete intersection ideal is the same as the Stanley depth of it's radical, see Theorem $2.1$. Also, we give some bounds for this invariant in Theorem $2.4$. 

\noindent
\textbf{Aknowledgements}. The author would like to express his gratitude to Asia Rauf for her help and for her encouragements.

\footnotetext[1]{This paper was supported by CNCSIS, ID-PCE, 51/2007}

\newpage
\section{Stanley depth of monomial irreducible ideals.}

\begin{lema}
Let $v_1,\ldots,v_m\in K[x_2,\ldots,x_m]$ be some monomials and let $a$ be a positive integer.
Let $I=(x_1^av_1,v_2,\cdots,v_m)$ and $I'=(x_1^{a+1}v_1,v_2,\cdots,v_m)$.
Then $\sdepth(I)=\sdepth(I')$.
\end{lema}

\begin{proof}
Let $g\in \mathbb N^n$ such that $x^g=LCM(x_1^av_1,v_2,\cdots,v_m)$, where $x^g=x_1^{g(1)}\cdots x_n^{g(n)}$. Obviously, $x^{g'}:=x_1x^g=LCM(x_1^{a+1}v_1,v_2,\cdots,v_m)$.
We denote $P=P_{I}^{g}$ and $P'=P_{I'}^{g'}$. According with \cite[Theorem 2.4]{hvz}, we choose a partition $\bigcup_{i=1}^{r}[c_i,d_i]$ of $P$, with $\sdepth(I)=\min\{\rho(d_i)|\;i\in[r]\}$. 
We will construct a partition $P'=\bigcup_{i=1}^{r}[c'_i,d'_i]$ with $\min\{\rho(d_i)|\;i\in [r] \} = \min\{\rho(d'_i)|\;i\in [r] \}$.
Therefore, $\sdepth(I')\geq \sdepth(I)$.

Suppose $\bigcup_{i=1}^{r}[c_i,d_i]$ is a partition is a partition of $P$. We define:
\[c'_i = \begin{cases} c_i,& if\;c_i(1)<a \\ c_i+e_1, & if\;c_i(1)=a \end{cases}
         \;and\; 
  d'_i = \begin{cases} d_i,& if\;d_i(1)<a-1 \\ d_i+e_1, & if\;d_i(1)\geq a-1 \end{cases}. \]
Note that $\rho(d_i)=\rho(d'_i)$ for all $i\in [r]$.
Indeed, $d'_i(1)=g'(1)=a+1$ if and only if $d_i(1)=g(1)=a$ and $d'_i(j)=g(j)$ if and only if $d_i(j)=g(j)$, for all $j\geq 2$.
We claim that $\bigcup_{i=1}^{r}[c'_i,d'_i]$ is a partition of $P'$.

Firstly, we show that $P'=\bigcup_{i=1}^{r}[c'_i,d'_i]$. Let $\sigma\in P'$. If $\sigma(1)\leq a$ it follows that $\sigma\in P$, thus
$\sigma\in [c_i,d_i]$ for some $i$. Indeed, if $\sigma\notin P$ then, since $\sigma\in P'$ it follows that $x_1^{a+1}v_1 |x^{\sigma}$ and therefore $\sigma(1)=a+1$, a contradiction.

If $\sigma(1)<a$, it follows that $c_i(1)<a$ and therefore $c'_i=c_i$. We get $c'_1 = c_i \leq d_i\leq d'_i$, thus $\sigma\in [c'_i,d'_i]$. 
If $\sigma(1)=a$, it follows that $d_i(1)=a$ and $d'_i(1)=a+1$. Suppose $\sigma\notin [c'_i,d'_i]$. Since $\sigma-e_1\in P$ (this is true, because $\sigma\in P'$, $\sigma(1)\neq a+1$ and thus $x_{1}^{a+1}v_1$ does not divide $x^{\sigma}$) it follows that $\sigma-e_1\in [c_j,d_j]$ for some $j\neq i$. But $(\sigma-e_1)(1)=a-1$ which implies $d_j(1)\geq a-1$. It follows that $\sigma \leq d_j+e_1=d'_j$ and therefore $\sigma\in [c'_j,d'_j]$, because $c'_j=c_j$.

If $\sigma(1)=a+1$, it follows that $\sigma-e_1\in P$ and therefore $\sigma-e_1\in[c_i,d_i]$ for some $i$. Indeed, if $x_1^{a+1}v_1|x^{\sigma}$ then $x_1^av_1|x^{\sigma-e_1}$, else, if $v_j|x^{\sigma}$ for some $j\geq 2$ then $v_j|x^{\sigma-e_1}$. But $(\sigma-e_1)(1)=a$ and therefore
$d_i(1)=a$. We get $\sigma\leq d_i + e_1 =d'_i$ and thus $\sigma\in [c'_i,d'_i]$.

Now, we must prove that for any $i\neq j$, we have $[c'_i,d'_i]\cap[c'_j,d'_j]=\emptyset$. Assume by contradiction, that there exists some $\sigma\in [c'_i,d'_i]\cap[c'_j,d'_j]$. We consider two cases.

(1) Case $d'_j = d_j$. If $\sigma(1)<a$ then $\sigma\leq d_i$ and $\sigma\geq c'_i=c_i$. As $d_j=d'_j\geq \sigma\geq c'_j = c_j$ we get $\sigma\in [c_i,d_i]\cap[c_j,d_j]=\emptyset$, which is false. If $\sigma(1)\geq a$ then $d'_j(1)\geq a$ and so $d_j(1)\geq a-1$. But then, $d'_j=d_j+e_1$, a contradiction.

(2) Case $d'_j>d_j$, $d'_i>d_i$. If $\sigma(1)<a$ then as before, $\sigma\leq d_i$, $\sigma\leq d_j$ and so $\sigma\in [c_i,d_i]\cap[c_j,d_j]=\emptyset$, which is false. If $\sigma(1)\geq a$ then $\sigma-e_1\leq d_i$, $\sigma-e_1\leq d_j$. On the other hand, if $c_i(1)< a$ then $\sigma(1)-1\geq c_i(1)$ and so $\sigma-e_1\geq c_i$, because $\sigma(j)\geq c'_i(j)=c_i(j)$ for $j>1$. Thus $\sigma-e_1\in [c_i,d_i]$. The same thing follows if $c_i(1)=a_1$ because $\sigma-e_1\geq c'_i - e_1 = c_i$. Similarly, $\sigma - e_1\in [c_j,d_j]$, which gives again a contradiction.

Suppose $\bigcup_{i=1}^{r}[c'_i,d'_i]$ is a partition of $P'$ with $\sdepth(I')=\min\{\rho(d'_i)|\; i\in [r]\}$. We define
\[c_i = \begin{cases} c'_i,&if\; c'_i(1)\leq a \\ c'_i-e_1, &if\; c'_i(1)=a+1 \end{cases}
         \;and\; 
  d_i = \begin{cases} d'_i,& if\;d'_i(1)<a \\ d'_i-e_1, & if\;d'_i(1)\geq a \end{cases}.\]
We may have $[c_i,d_i]=\emptyset$ for some indexes $i$. We claim that $\bigcup_{[c_i,d_i]\neq \emptyset}[c_i,d_i]$ is a partition of $P$. Suppose the claim is true. Note that $\rho(d_i)=\rho(d'_i)$ for all $1\leq i\leq m$. It follows that \[ \min\{\rho(d'_i)|i\in[r]\} \leq \min\{\rho(d_i)|i\in[r]\; and \; [c_i,d_i]\neq \emptyset \} \] and thus, $\sdepth(I')\leq \sdepth(I)$.

Now, we prove the claim. Let $\sigma\in P$. If $\sigma(1)<a$, it follows that $\sigma\in P'$ and therefore $\sigma\in [c'_i,d'_i]$ for some $i$. Indeed, since $\sigma(1)<a$ it follows that $x^av_1$ does not divide $x^{\sigma}$ and therefore, $v_j|x^{\sigma}$ for some $j>1$. Since $\sigma(1)<a$ we get $c'_i(1)<a$ and therefore $c_i=c'_i$. If $d'_i(1)\geq a$ it follows $d_i=d'_i-e_1$, otherwise $d_i=d'_i$. In both cases, $\sigma(1)\leq d_i(1)$. Thus $\sigma\in [c_i,d_i]$.

Suppose $\sigma(1)=a$. We have $\sigma+e_1\in P'$, because $(\sigma+e_1)(1)=a+1$. Therefore $\sigma+e_1 \in [c'_i,d'_i]$ for some $i$. It follows that $d'_i(1)=a+1$ and therefore $d_i=d'_i-e_1$. On the other hand, $\sigma(1)\leq a$. Thus, we get $\sigma\in [c_i,d_i]$.

Now, we must prove that for any $i\neq j$, we have $[c_i,d_i]\cap[c_j,d_j]=\emptyset$. If $c_i=c'_i$ and $c_j=c'_j$ it follows that $[c_i,d_i]\subset [c'_i,d'_i]$ and $[c_j,d_j]\subset [c'_j,d'_j]$. Thus, there is nothing to prove. We can assume $c_i=c'_i-e_1$. In this case, it follows that $c_i(1)=d_i(1)=a$ and $[c_i,d_i]=[c'_i-e_1,d'_i-e_1]$. If $d'_j(1)<a+1$, then $d_j(1)<a$ and thus $[c_i,d_i]\cap[c_j,d_j]=\emptyset$. Otherwise, $d'_j(1)=a+1$ and $d_j(1)=a$. Let $\sigma \in [c_i,d_i]$. Obviously, $\sigma+e_1\in [c'_i,d'_i]$. On the other hand, if $\sigma \in [c_j,d_j]$, then $\sigma+e_1\in [c'_j,d'_j]$, a contradiction.
\end{proof}

The following theorem is a direct consequence of \cite[Theorem 1.1]{par} and \cite[Theorem 2.4]{hvz}.

\begin{teor}
$\sdepth((x_1,\ldots,x_n))=\left\lceil \frac{n}{2} \right\rceil$.
\end{teor}

As a consequence of Lemma $1.1$ and Theorem $1.2$, we get.

\begin{teor}
Let $a_1,\ldots,a_n$ be some positive integers. Then:
\[ \sdepth((x_1^{a_1},\ldots,x_n^{a_n})) = \sdepth((x_1,\ldots,x_n)) = \left\lceil \frac{n}{2} \right\rceil. \]
In particular, $\sdepth((x_1^{a_1},\ldots,x_m^{a_m})) = n - m + \left\lceil \frac{m}{2}\right\rceil$ for any $1\leq m\leq n$.
\end{teor}

\begin{proof}
We apply Lemma $1.1$ several times.
The second part of the theorem, follows from \cite[Proposition 3.6]{hvz}.
\end{proof}

\newpage
\section{Stanley depth of monomial complete intersections}

\begin{teor}
Let $I\subset S$ be a complete intersection monomial ideal. Then $\sdepth(I)=\sdepth(\sqrt{I})$.
\end{teor}

\begin{proof}
Let $I=(v_1,\ldots,v_m)\subset S$ be a complete intersection monomial ideal. If $I$ is square free, then there is nothing to prove. Otherwise, we may assume $x_1^2|v_1$. By Lemma $1.1$, $\sdepth(I)=\sdepth((v_1/x_1,v_2,\ldots,v_m))$. Thus, we can replace $I$ with
$(v_1/x_1,v_2,\ldots,v_m)$ and than we apply the previous step. After a finite number of steps, we must stop. The last ideal obtained with this algorithm will be $\sqrt{I}$.
\end{proof}

Now, we prove a version of a result of Sumiya Nasir \cite[Corollary 3.2]{sum} for ideals.

\begin{lema}
Let $I'\subset S[x_{n+1}]$ be a monomial ideal. We consider the homomorphism $\varphi:S[x_{n+1}]\rightarrow S$, $\varphi(x_i)=x_i$ for $i\leq n$ and $\varphi(x_{n+1})=1$. Let $I=\varphi(I')$. Then $\sdepth(I')\leq \sdepth(I)+1$.
\end{lema}

\begin{proof}
We choose a Stanley
decomposition $\mathcal{D}':\,\,I'=\bigoplus_{i=1}^ru_iK[Z_i']$ of $I'$
with $\sdepth (\mathcal{D}')=\sdepth (I')$. We claim that
$I = \bigoplus_{{x_{n+1}}\in
Z_i'}u_i/{x_{n+1}}^{a_i}K[Z_i'\setminus\{x_{n+1}\}],$ where $a_i$ is
maximum integer such that ${x_{n+1}}^{a_i}$ divides $u_i$.

We consider a monomial $w\in \sum_{{x_{n+1}}\in
Z_i'}u_i/{x_{n+1}}^{a_i}K[Z_i'\setminus\{x_{n+1}\}]$, that is, there
exists $i\in[r]$ such that $w\in
u_i/{x_{n+1}}^{a_i}K[Z_i'\setminus\{x_{n+1}\}]$. It follows that
${x_{n+1}}^{a_i}w\in u_iK[Z_i'\setminus\{x_{n+1}\}]$. So ${x_{n+1}}^{a_i}w\in I'$. Since $x_{n+1}$ does not divide $w$ it follows that $w \in I$.

In order to prove other inclusion we consider a monomial $u\in I$.
Since $I=\varphi(I')$ we have $x_{n+1}^t u\in I'$ for sufficiently large $t$. Then $x_{n+1}^t u\in u_i K[Z'_i]$ for some $i\in [r]$.
If $x_{n+1}\in Z'_i$ then we get $u\in u_i/x_{n+1}^{a_i}K[Z'_i\setminus \{x_{n+1}\}]$. Otherwise increasing $t$ we may suppose that
$x_{n+1}^t u\in u_j K[Z'_j]$ for some $j\in[r]$ with $x_{n+1}\in Z'_j$ and we may continue as above.

Now, we prove
that this sum is direct. Let \[ v\in
u_i/{x_{n+1}}^{a_i}K[Z_i'\setminus\{x_{n+1}\}]\cap
u_j/{x_{n+1}}^{a_j}K[Z_j'\setminus\{x_{n+1}\}] \] be a monomial for
$i\neq j$. We have
$v=u_i/{x_{n+1}}^{a_i}f_i=u_j/{x_{n+1}}^{a_j}f_j$ where $f_i\in
K[Z_i'\setminus\{x_{n+1}\}]$ and $f_j\in K[Z_j'\setminus\{x_{n+1}\}]$
are monomials. We obtain $vx_{n+1}^{a}\in u_iK[Z_i']\cap u_jK[Z_j']$
where $a=\max\{a_i,a_j\}$, since $x_{n+1}\in K[Z_i'],K[Z_j']$ and thus, we get a
contradiction. Hence \linebreak $I=\bigoplus_{{x_{n+1}}\in
Z_i'}u_i/{x_{n+1}}^{a_i}K[Z_i'\setminus\{x_{n+1}\}]$ and therefore
$\sdepth(I)+1\geq \sdepth(I')$.
\end{proof}

\begin{cor}
Let $I\subset S$ be a monomial ideal with
$G(I)=\{v_1,\ldots,v_m\}$. Let $S'=S[x_{n+1}]$ and
$I'=(v_1,\ldots,v_{m-1},x_{n+1}v_m).$ Then $\sdepth(I')\leq \sdepth(I)+1$.
\end{cor}

\begin{teor}
Let $I\subset S$ be a monomial complete intersection ideal. Then:
\[ 1 + n - m \leq \sdepth(I) \leq \left\lceil \frac{m}{2} \right\rceil + n - m.\]
In particular, $I$ satisfy the Stanley's conjecture.
\end{teor}

\begin{proof}
The first inequality is a particular case of \cite[Proposition 3.4]{hvz}. Moreover, $\depth(I)=n-m+1$ and thus $\sdepth(I)\geq \depth(I)$.
In order to prove the second assertion, we consider a monomial complete intersection ideal $I\subset S$. By Theorem $2.1$, we can assume that $I$ is square free. Let $I=(v_1,\ldots,v_m)$, where
$v_1,\ldots,v_m$ is a regular sequence of squarefree monomials in $S$. We use induction on $n$. The case $n=1$ is obvious. If $I$ is generated by variables, then Theorem $1.3$ implies $\sdepth(I)= \left\lceil \frac{m}{2} \right\rceil + n - m$. Otherwise, we can assume that $x_n|v_m$ and $supp(v_m)\setminus \{x_n\}\neq \emptyset$. We denote $J=(v_1,\ldots,v_m/x_n)\cap K[x_1,\ldots,x_{n-1}]$. 
By induction, it follows that $\sdepth(J)\leq \left\lceil \frac{m}{2} \right\rceil + n - m -1$. Thus, Corollary $2.3$ implies $\sdepth(I)\leq \sdepth(J)+1 \leq \left\lceil \frac{m}{2} \right\rceil + n - m$, as required.
\end{proof}

\begin{conj}
If $I\subset S$ is a monomial complete intersection ideal, then $\sdepth(I)=\left\lceil \frac{m}{2} \right\rceil + n - m$.
\end{conj}

Herzog, Vladoiu and Zheng show in \cite{hvz} that this conjecture is true for $n\leq 3$.

\begin{exm}
Let $I=(x_1,x_2)\subset K[x_1,x_2]$. With the notations of Corollary $2.3$, we have $I'=(x_1,x_2x_3)\subset K[x_1,x_2,x_3]$. We consider the following Stanley decomposition of $I'$:
\[ I' = x_2x_3 K[x_1,x_2,x_3]\oplus x_1K[x_1,x_3] \oplus x_1x_2 K[x_1,x_2].\]
Then $I=x_2K[x_1,x_2]\oplus x_1K[x_1]$ is a Stanley decomposition of $I$.
\end{exm}

We recall some definitions from \cite{hsy}. A Stanley space $uK[Z]$ is called squarefree, if $u\in S$ is a squarefree monomial and $supp(u)\subset Z$. A Stanley decomposition of $I$ is called squarefree if all its Stanley spaces are squarefree. We shall use the following notation: for $F\subset [n]$, we set $x_F=\prod_{i\in F}x_i$ and $Z_F=\{x_i|i\in F\}$.

\begin{obs}
\emph{Let $I\subset S$ be a squarefree monomial ideal. Then $I$ has squarefree Stanley decompositions. For instance, $\bigoplus_{F\subset [n],x_F\in I} x_F K[Z_F]$ is a squarefree Stanley decomposition for $I$. Moreover, if $\mathcal D: I = \bigoplus_{i=1}^r u_iK[Z_i]$ is a Stanley decomposition for $I$, then, by the proof of \linebreak \cite[Theorem 2.4]{hvz}, $\mathcal D': I = \bigoplus_{u_i\;squarefree} u_iK[Z_i\cup supp(u_i)]$ is a (squarefree) Stanley decomposition for $I$. Note that $\sdepth(\mathcal D')\geq \sdepth(\mathcal D)$ and therefore 
\[\sdepth(I) = \max\{\sdepth(\mathcal D)|\;D\;is\;a\;squarefree\;Stanley\;decomposition\;for\;I \}.  \]
An algorithm to compute a squarefree Stanley decomposition for an arbitrary squarefree monomial ideal $I\subset S$ was given by Imran Anwar in \cite[Lemma 2.2]{imran}.}
\end{obs}

\begin{obs}
\emph{Let $I\subset S$ be a squarefree monomial ideal with $G(I)=\{v_1,\ldots,v_m\}$. Let $I'=(v_1,\ldots,v_{m-1},v_my)\subset S[y]$.
Suppose $I=\bigoplus_{i=1}^r u_i K[Z_i]$ is a squarefree Stanley decomposition of $I$. One can easily see that
$IS[y]=\bigoplus_{i=1}^r u_i K[Z_i\cup \{y\}]$. Since $I'\subset IS[y]$, it follows that
\[ I' = \bigoplus_{i=1}^r (u_i K[Z_i\cup \{y\}]\cap I'). \]
We can refine this direct sum to a squarefree Stanley decomposition $\mathcal D'$ of $I'$. As the following example shows, the decomposition of $I'$ obtained may have $\sdepth(\mathcal D')<\sdepth(I')$.}
\end{obs}

\begin{exm}
Let $I=(x_1,x_2,x_3)\subset K[x_1,x_2,x_3]$ and $I'=(x_1,x_2,x_3x_4)\subset K[x_1,x_2,x_3,x_4]$. We consider the following (squarefree) Stanley decomposition of $I$:
\[ I = x_1x_2x_3K[x_1,x_2,x_3]\oplus x_1K[x_1,x_3]\oplus x_2K[x_1,x_2]\oplus x_3K[x_2,x_3]. \]
Since $x_1 \in I'$ it follows that $x_1x_2x_3K[x_1,x_2,x_3,x_4]\cap I' = x_1x_2x_3K[x_1,x_2,x_3,x_4]$. Similarly, $x_2K[x_1,x_2,x_4]\cap I' =x_2K[x_1,x_2,x_4]$ and $x_1K[x_1,x_3,x_4]\cap I' =x_1K[x_1,x_3,x_4]$. Also, $x_3K[x_2,x_3,x_4]\cap I' = x_3x_4K[x_2,x_3,x_4]\oplus x_2x_3K[x_2,x_3]$.
Thus, $I' = x_1x_2x_3K[x_1,x_2,x_3,x_4]\oplus x_2K[x_1,x_2,x_4] \oplus x_1K[x_1,x_3,x_4] \oplus x_3x_4K[x_2,x_3,x_4]\oplus x_2x_3K[x_2,x_3]$. On the other hand, by \cite[Proposition 3.8]{hvz}, we have $\sdepth(I')=3$.
\end{exm}


\vspace{2mm} \noindent {\footnotesize
\begin{minipage}[b]{15cm}
 Mircea Cimpoeas, Institute of Mathematics of the Romanian Academy, Bucharest, Romania\\
 E-mail: mircea.cimpoeas@imar.ro
 
 \end{minipage}}

\end{document}